\documentclass[11pt]{article}

\usepackage{amsmath, amsthm, amssymb, latexsym, graphicx,courier,makeidx,subeqnarray}

\usepackage{color}

\newtheorem{theorem}{Theorem}[section]
\newtheorem{definition}[theorem]{Definition}
\newtheorem{lemma}[theorem]{Lemma}

\newtheorem{problem}[theorem]{Problem}
\newtheorem{corollary}{Corollary}[section]
\newtheorem{remark}{Remark}[section]

\begin{document}
	
\title{Inverse Problems for Nonlinear Quasi-Variational Inequalities with an Application to Implicit Obstacle Problems of $p$-Laplacian Type
\thanks{\, Project supported by the European Union's Horizon 2020 Research and Innovation Programme under the Marie Sk{\l}odowska-Curie grant agreement No. 823731 CONMECH, National Science Center of Poland under Maestro Project No. UMO-2012/06/A/ST1/00262, and
National Science Center of Poland under Preludium Project No. 2017/25/N/ST1/00611.
It is also supported by the Beibu Gulf University, Project No. 2018KYQD06, Natural Science Foundation of Guangxi (Grant No. 2018JJA110006), and
the International Project co-financed by the Ministry of Science and Higher Education of Republic of Poland under Grant No. 3792/GGPJ/H2020/2017/0.}}

\author{
Stanis{\l}aw Mig\'orski \footnotemark[2] \and
Akhtar A. Khan \footnotemark[3]
\and
Shengda Zeng \footnotemark[4]}
\renewcommand{\thefootnote}{\fnsymbol{footnote}}
\footnotetext[2]{\,College of Applied Mathematics, Chengdu University of Information Technology, Chengdu,  610225, Sichuan Province, P.R. China, and Jagiellonian Uni\-ver\-sity in Krakow, Chair of Optimization and Control, ul. Lojasiewicza 6, 30348 Krakow, Poland.
%%Tel.: +48-12-6646666.
({\tt sta\-nis\-law.migorski@uj.edu.pl})}
\footnotetext[3]{\, Center for Applied and Computational Mathematics, School of Mathematical Sciences,
Rochester Institute of Technology, 85 Lomb Memorial Drive, Rochester, New York, 14623, USA.
({\tt aaksma@rit.edu})}
\footnotetext[4]{\, Jagiellonian University in Krakow, Faculty of Mathematics and Computer Science, ul. Lojasiewicza~6, 30348 Krakow, Poland. Corresponding author.
%%Tel.: +86-18059034172.
%%E-mail address:	
({\tt shengdazeng@gmail.com, shdzeng@hotmail.com, zeng\-shengda@163.com})}

\renewcommand{\thefootnote}{\fnsymbol{footnote}}

\date{}
\maketitle

\begin{abstract}\noindent
The primary objective of this research is to investigate an inverse problem of parameter identification in nonlinear mixed quasi-variational inequalities posed in a Banach space setting. By using a fixed point theorem, we explore properties of the solution set of the considered quasi-variational inequality. We develop a general regularization framework to give an existence result for the inverse problem. Finally, we apply the abstract framework to a concrete inverse problem of identifying the material parameter in an implicit obstacle problem given by an operator of $p$-Laplacian type.


\end{abstract}

\section{Introduction}\label{Introduction}
%%%%%%%%%%%%%%%%%%%%%%
Variational inequalities provide a powerful mathematical tool to explore a broad spectrum of vi\-tal problems such as obstacle problems, unilateral contact problems, optimization and control problems, traffic network models, equilibrium problems, and many others, see~\cite{liu2,liu1,sofonea1,tang1}.
In the study of classical variational inequalities, the constraint set remains independent of the state variable.
However, in many important situations arising in engineering and economic models, such as Nash equilibrium problems with shared constraints and transport optimization feedback control problems, the constraint sets directly depend on the unknown state variable.
This leads naturally to the notion of a quasi-variational inequality. Recently, numerous authors have contributed to strengthening the theory and applicability of quasi-variational inequalities.
In the following, we provide a brief review of some of the recent developments in this direction.
Khan-Motreanu~\cite{khan1} gave new existence results for elliptic and evolutionary variational and quasi-variational inequalities by using Mosco-type continuity properties and a fixed point theorem for set-valued maps. Liu-Zeng~\cite{liu3} studied optimal control of generalized quasi-variational hemivariational inequalities involving multivalued mapping. Liu-Motreanu-Zeng~\cite{liu4} investigated a notion of well-posedness for differential mixed quasi-variational inequalities in Hilbert spaces. Aussel-Sultana-Vetrivel~\cite{aussel} established existence results for the projected solution of quasi-variational inequalities in a finite-dimensional setting. Khan-Tammer-Zalinescu~\cite{khan2} employed the elliptic regularization technique to study an ill-posed quasi-variational inequality with contaminated data, and showed that a sequence of bounded regularized solutions converges strongly to a solution of the original quasi-variational inequality.

In the present paper,
our first goal is to study existence of solution to the nonlinear mixed quasi-variational inequality in the following functional framework.
Let $V$ be a real reflexive Banach space
and $C$ be a nonempty, closed, and convex subset of $V$. Let $B$ be another Banach space, and  $A\subseteq B$ be the set of admissible parameters. Given a set-valued map
$K\colon C\to 2^C$, a nonlinear map
$T\colon B\times V\to V^*$,
a functional $\varphi\colon V\to \overline{\mathbb{R}}
:=\mathbb{R}\cup\{+\infty\}$, and $m\in V^*$, the problem reads as follows: find $u\in C$ such that $u\in K(u)$ and
\begin{equation}\label{NEW1}
\langle T(a,u),v-u\rangle +\varphi(v)-\varphi(u)\ge \langle m,v-u\rangle
\ \ \mbox{for all} \ \ v\in K(u).
\end{equation}
The second goal of this paper is to investigate  the inverse problem of identifying a variable parameter $a$ from the measured data $z$ such that a solution $u(a)$ of the mixed quasi-variational inequality
(\ref{NEW1})
is closest to the data $z$ in some norm.
We will study this inverse problem in an optimization framework which is most suitable for incorporating a regularization process. We note that the regularization is necessary due to the severely ill-posed nature of the inverse problem.

We note that today the inverse problem of parameter identification in partial differential equations is an important and mature subject. However, in recent years, motivated by various applications, inverse problems in variational and quasi-variational inequalities have attracted a lot of attention. To mention a few of the recent contributions, we note that Gwinner-Jadamba-Khan-Sama~\cite{gwinner1} examined the inverse problem in an optimization setting using the output-least squares formulation, and obtained existence as well as convergence results for the optimization problem. Clason-Khan-Sama-Tammer~\cite{khan3} explored the inverse problem of parameter identification in noncoercive variational problems via the output least-squares and the modified output least-squares objectives.  Gwinner~\cite{gwinner2} focused on the inverse problem of parameter identification in variational inequalities of the second kind that does not only treat the parameter linked to a bilinear form but importantly also the parameter linked to a nonlinear non-smooth function. For more details on this topic, the reader is referred to Alleche-R$\breve{\mbox{a}}$dulescu~\cite{radulescu3,radulescu2,radulescu1}, Aussel-Gupta-Mehra~\cite{aussel2}, Gwinner~\cite{gwinner3}, Khan-Mig\'orski-Sama~\cite{KMS}, Khan-Sama~\cite{khan4},  Kassay-R\u{a}dulescu~\cite{KassayRadulescu2018},  and the references therein.

Quite recently, Khan-Motreanu~\cite{khanmotreanu} studied the following inverse problem of parameter identification driven by a quasi-variational inequality: find $a\in A$ by solving
\begin{equation}\label{KMproblem}
\min_{a\in A}J_\kappa(a):=\frac{1}{2}\|u(a)-z\|_Z^2
+\kappa R(a),
\end{equation}
where $\kappa>0$ is a regularization parameter, $R$ is a regularization operator,
$z\in Z$ is the measured data,
and $u(a)$ is the unique solution to the following quasi-variational inequality associated with the parameter $a$: find $u(a)\in K(u(a))$ such that
\begin{equation}\label{KMproblem2}
T(a,u,v-u)\ge \langle m,v-u\rangle
\ \ \mbox{for all}\ \ v\in K(u(a)).
\end{equation}
The functional setup of~\cite{khanmotreanu} is described as follows:
$T\colon B\times V\times V\to \mathbb R$
is a trilinear bounded function such that
$T(a,u,v)$ is symmetric in $u$ and $v$,
and
there exist constants $\alpha>0$ and $\beta>0$ satisfying
\begin{eqnarray*}
&&T(a,u,v)\le \beta\|a\|_B\|u\|_V\|v\|_V
\ \ \mbox{for all}\ \ u, v\in V,\ a \in B,
\\[2mm]
&&T(a,u,u)\ge \alpha\|u\|_V^2
\ \ \mbox{for all}\ \ u \in V,\ a \in A.
\end{eqnarray*}
The authors in~\cite{khanmotreanu} first proved the existence of a global minimizer and gave
convergence results for the optimization problem  (\ref{KMproblem}).
Next, they discretized the identification problem and provided the convergence analysis for
the discrete problem. Finally, an application to the gradient obstacle problem was given.

Here, we have to point out that under the assumption that $T$ is trilinear and strongly monotone, it is easy to prove the quasi-variational inequality (\ref{KMproblem2}) has a solution. However, there are a lot of identification problems driven by mixed quasi-variational inequalities (i.e., quasi-variational inequalities which involve convex and lower semicontinuous functionals as in (\ref{NEW1})) and not by a quasi-variational inequality.
Furthermore, if the solution set for a mixed quasi-variational inequality corresponding to an identification problem is not a singleton, then the difficulty of research increases since we need to analyze various properties of the solution mapping (which is a set-valued map), such as continuity, etc. The present paper deals with a generalized and complicated identification problem under the more general functional framework.
More precisely, the paper is built on~\cite{khanmotreanu} and extends results given there in the following three directions.
First, we consider the quasi-variational inequality for a nonlinear and not necessarily strongly monotone map whereas the framework developed in~\cite{khanmotreanu} is limited to a bilinear elliptic form. Second, we conduct the study in a Banach space setting, not in a Hilbert space.
The adopted generality allows applying our results to an implicit obstacle problem involving the operator of $p$-Laplace type. Third, the inclusion of the functional $\varphi$ enhances the applicability of our abstract framework.

The paper is organized as follows.
In Section~\ref{Section2}, we provide the primary results of the paper. They include an existence result for the nonlinear quasi-variational inequality and an existence result for inverse problem. Section~\ref{Section3} provides
an application of the results to an implicit obstacle problem of $p$-Laplacian type.

\section{Main Results}\label{Section2}
In this section, we explore the properties of the solution set for the quasi-variational inequality under consideration and provide an existence result for the inverse problem.

\subsection{Solvability of the Nonlinear Quasi-Variational Inequality}

Let $V$ be a reflexive Banach space equipped with the norm $\|\cdot\|_V$ and
$\langle \cdot, \cdot \rangle$ be the duality pairing between $V$ and its dual $V^*.$
The weak and the norm convergences are denoted by $\rightharpoonup$ and $\to$, respectively.
Let $C$ be a nonempty, closed and convex subset of $V$. Let $B$ denote another Banach space and
$A$ be a subset of $B$.

We consider the following nonlinear mixed quasi-variational inequality (MQVI):
given $a\in A$,
find $u\in C$ with $u\in K(u)$ such that
\begin{equation}\label{mqvi}
\langle T(a,u),v-u\rangle +\varphi(v)-\varphi(u)\ge \langle m,v-u\rangle
\ \ \mbox{for all} \ \ v\in K(u).
\end{equation}

We introdude the following hypotheses on the data of inequality (\ref{mqvi}).
\begin{itemize}
\item[($\mathcal{H}_T$).]
The mapping
$T\colon B\times V\to V^* $ is bounded and such that
\begin{itemize}
\item[(i)] For each $u\in V$, the mapping
$T(\cdot, u) \colon B \to V^*$ is linear.
\item[(ii)] For each $a\in A$, the mapping $T(a,\cdot)\colon V\to V^*$ is monotone and continuous.
\end{itemize}
\item[($\mathcal{H}_{\varphi}$).]
The functional $\varphi\colon V\to \overline{\mathbb{R}}:=\mathbb{R}\cup\{+\infty\}$ is proper, convex, and lower semicontinuous with  $C\subset \mbox{int(dom}(\varphi))$.
\item[($\mathcal{H}_K$).]
The set-valued mapping $K\colon C\to 2^C$ is such that for all $u\in C$, the set $K(u)\subseteq C$ is nonempty, closed, convex, and
\begin{itemize}
\item[(i)] For any sequence $\{x_n\}\subset C$ with $x_n\rightharpoonup x$, and for any $y\in K(x),$ there exists a sequence $\{y_n\}\subset C$ such that $y_n\in K(x_n)$ and $y_n\to y,$ as $n\to\infty.$
\item[(ii)] For any sequences $\{x_n\}$ and $\{y_n\}$ in $C$ with $y_n\in K(x_n)$, if $x_n\rightharpoonup x$ and $y_n\rightharpoonup y$, then $y\in K(x)$.
\end{itemize}
\item[($\mathcal{H}_0$).] There is a bounded subset $C_0$ of $V$ with $K(u)\cap C_0\neq \emptyset,$ for each $u\in C.$ Furthermore, there exists a function $h:\mathbb{R}_+\to \mathbb{R}$ with $h(t)\to +\infty$ as $t\to +\infty$ such that
\begin{equation}\label{coercive}
\langle T(a,w),w-v_0\rangle
+\varphi(u)-\varphi(v_0)
\ge h(\|w\|_V)\, \|w\|_V
\ \,\mbox{\rm for all}
\ v_0 \in C_0,\, a\in A,\, w\in C.
\end{equation}
\end{itemize}

\begin{remark}
\rm
Note that the coercivity hypothesis $(\mathcal{H}_0)$ is uniform with respect to $a\in A$.
Evidently, if $0\in K(u)$ for all $u\in C$, and $T(a,w)=T(w)$, then condition \eqref{coercive} reduces to the classical coercivity condition.  Assumption ($\mathcal{H}_K$) usually means
that the set-valued mapping $K\colon C\to 2^C$
is $M$-continuous, see Definition~\ref{defap}.
\end{remark}

We recall the following fixed point theorem by Kluge~\cite{Kluge}.
\begin{theorem}\label{fpt}
Let $Z$ be a reflexive Banach space and
$C\subset Z$ be nonempty, closed and convex. Assume that $\Psi\colon C\to 2^C$ is a set-valued mapping such that for every $u\in C$, the set $\Psi(u)$ is nonempty, closed, and convex, and the graph of $\Psi$ is sequentially weakly closed. If either $C$ is bounded or $\Psi(C)$ is bounded, then the map $\Psi$ has at least one fixed point in $C$.
\end{theorem}

The following Minty-type lemma represents an equivalent formulation of (MQVI).
\begin{lemma}\label{Minty}
Assume that $(\mathcal{H}_T)$ and $(\mathcal{H}_{\varphi})$ hold,
and $m \in V^*$. If $K\colon C\to 2^C$ is such that for each $u\in C$, $K(u)\subseteq C$ is nonempty, closed, and convex,  then (MQVI) is equivalent to the following inequality: given $a\in A$,
find $u\in C$ such that $u\in K(u)$ and
\begin{equation}\label{mmqvi}
\langle T(a,v),v-u\rangle +\varphi(v)-\varphi(u)\ge \langle m,v-u\rangle
\ \ \textrm{\rm for all}\ \ v\in K(u).\end{equation}
\end{lemma}
\begin{proof}
Let $u\in C$ be a solution to inequality (MQVI). Then $u\in K(u)$ and
\begin{eqnarray*}
\langle T(a,u),v-u\rangle +\varphi(v)-\varphi(u)\ge \langle m,v-u\rangle
\ \  \textrm{\rm for all}\ \ v\in K(u).
\end{eqnarray*}
Using the monotonicity of $T(a,\cdot)$,
we immediately get that $u$ is also a solution
to (\ref{mmqvi}).

Conversely, let $u\in C$ be a solution to  (\ref{mmqvi}). Since $K(u)$ is convex, for all  $w\in K(u)$ and $\lambda\in(0,1)$, we take
$v= u_\lambda:=\lambda w+(1-\lambda)u\in K(u)$
in (\ref{mmqvi}) to get
\begin{equation*}
\langle T(a,u_\lambda),w-u\rangle +\varphi(w)-\varphi(u)\ge \langle m,w-u\rangle.
\end{equation*}
The hemicontinuity of the mapping $T(a,\cdot)$ confirms that $u\in K(u)$ is a solution to (MQVI). \end{proof}

In what follows, the solution set to inequality (MQVI) corresponding to a parameter $a$
is denoted by $\Gamma(a)$.
The main result of this subsection is the following existence result.
\begin{theorem}\label{maintheorem1}
Assume that $(\mathcal{H}_T)$, $(\mathcal{H}_{\varphi})$, $(\mathcal{H}_K)$, and $(\mathcal{H}_0)$ hold, and $m \in V^*$.
Then for each $a\in A$ fixed, the set $\Gamma(a)$ is nonempty, bounded, and weakly closed. \end{theorem}
\begin{proof}
We will first show that
$\Gamma(a)\ne \emptyset$ for all $a \in A$.	
To this goal, we exploit the commonly used technique of finding a fixed point of the variational selection.
Let $a\in A$ be arbitrary but fixed parameter.
For a fixed $w\in C$, we consider the variational inequality: find $u\in K(w)$ such that
\begin{equation}\label{mvi}
\langle T(a,u),v-u\rangle+\varphi(v)-\varphi(u)\ge \langle m, v-u\rangle
\ \ \textrm{for all}\ \ v\in K(w).\end{equation}
Define the variational selection
$S_a\colon C\to 2^C$ that to any $w\in C$ associates the set of solutions to inequality (\ref{mvi}), namely,
\begin{eqnarray*}
S_a(w):=\big\{u\in K(w) \mid
\mbox{$u$ solves the problem (\ref{mvi})}\big\}.
\end{eqnarray*}
It is clear that any fixed point of set-valued mapping $S_a$
is a solution to inequality (MQVI). The proof that
$\Gamma(a)\ne \emptyset$ for all $a \in A$ is based on showing that the variational selection $S_a$ satisfies the assumptions imposed on the map $\Phi$ in Theorem~\ref{fpt}.

First, using (\ref{coercive}) and hypotheses $(\mathcal{H}_{\varphi})$,
$(\mathcal{H}_0)$,
it follows from \cite[Theorem 3.2]{liuzengmigorski} that for every
$a\in A$ and $w\in C,$ the set $S_a(w)$ is nonempty, closed, and convex.

Second, we claim that for any $a\in A$,
the graph of $S_a$ is sequentially weakly closed. Let $\{w_n\}$,
$\{u_n\}$ be sequences such that
$u_n\in S_a(w_n)$ with
$u_n\rightharpoonup u$ and
$w_n\rightharpoonup w$, as $n\to \infty$. Then $u_n\in K(w_n)$ and using Lemma~\ref{Minty}, we have
\begin{equation}\label{eq1}
\langle T(a,v),v-u_n\rangle+\varphi(v)-\varphi(u_n)\ge \langle m,v-u_n\rangle
\ \ \textrm{for all}\ \ v\in K(w_n).
\end{equation}
Note that $(w_n,u_n)\in\mbox{graph}(K)$ with $(w_n,u_n)\rightharpoonup(w,u)$ in $C\times C$,  so by hypothesis $(\mathcal{H}_K)$\,(ii),
we have $u\in K(w)$.
On the other hand, for any $z\in K(w),$  by condition $(\mathcal{H}_K)$\,(i),
there exists a sequence $\{v_n\}\subset C$
with $v_n\in K(w_n)$ for all $n\in \mathbb N$ such that $v_n\to z$, as $n\to \infty$.
Recall that $C\subset \mbox{int(dom}\varphi))$, so by invoking~\cite[Proposition 2.2]{barbu}, the function $\varphi$ is continuous on $C$.  Inserting $v=v_n$ into (\ref{eq1}), and passing to the upper limit, as $n\to \infty$, we obtain
\begin{align*}
\langle T(a,z),z-u\rangle+\varphi(z)-\varphi(u)&\ge\limsup_{n\to\infty}\langle T(a,v_n)-T(a,z),v_n-u_n\rangle+\limsup_{n\to\infty}\langle T(a,z),v_n-u_n\rangle\\
&\quad+\limsup_{n\to\infty}\varphi(v_n)-\liminf_{n\to\infty}\varphi(u_n)\\
&\ge\limsup_{n\to\infty}\langle T(a,v_n),v_n-u_n\rangle+\limsup_{n\to\infty}\varphi(v_n)-\liminf_{n\to\infty}\varphi(u_n)\\
&\ge \limsup_{n\to\infty}\bigg(\langle T(a,v_n),v_n-u_n\rangle+\varphi(v_n)-\varphi(u_n)
\bigg)\\
&\ge \limsup_{n\to\infty}\langle m,v_n-u_n\rangle=\langle m,z-u\rangle
\ \ \textrm{for all}\ \ z\in K(w).
\end{align*}
By Lemma~\ref{Minty}, it follows that
$u\in S_a(w)$, which implies that the graph of  $S_a$ is sequentially weakly closed.

Third, we claim that the set $S_a(C)$ is bounded. Arguing by contradiction, suppose that $S_a(C)$ is unbounded, and there are sequences
$\{w_n\}$ and $\{u_n\}$ such that
$u_n\in S_a(w_n)$ and $\|u_n\|_V\to\infty,$ as $n\to\infty$.
Therefore, $u_n\in K(w_n)$ and
\begin{equation*}
\langle T(a,u_n),v-u_n\rangle+\varphi(v)-\varphi(u_n)\ge \langle m,v-u_n\rangle
\ \ \textrm{for all}\ \ v\in K(w_n).
\end{equation*}
By hypothesis $(\mathcal{H}_0)$, there is a sequence $\{v_n\}\subset C$ such that
$v_n\in C_0\cap K(w_n),$ for each $n\in\mathbb N$.
We set $v=v_n$ in the above inequality and rearrange the resulting inequality to obtain
$$
\langle m,u_n-v_n\rangle\geq \langle T(a,u_n),u_n-v_n\rangle+\varphi(u_n)-\varphi(v_n)
\geq h(\|u_n\|_V)\|u_n\|_V,
$$
where $h\colon \mathbb{R}_+\to \mathbb{R}$ with $h(r)\to +\infty$, as $r\to+\infty.$ The above inequality implies that
\begin{equation*}
h(\|u_n\|_V)\leq \|m\|_{V^*}
\Big(1+\frac{\|v_n\|_V}{\|u_n\|_V}\Big),
\end{equation*}
and by passing to the limit as $n\to \infty$,
we get a contradiction.
Hence $S_a(C)$ is bounded set.

Therefore, all the conditions of Theorem~\ref{fpt} have been verified for the set-valued mapping
$S_a$
and hence it has a fixed point.
Consequently, for each $a\in A$, we have $\Gamma(a)\ne \emptyset$.
Note that since $\Gamma(a)\subset S_a(C)$,
for all $a\in A$, the set $\Gamma(a)$ is bounded as well.

Finally, it remains to prove that
for each $a\in A$, the set $\Gamma(a)$
is weakly closed.
Let $\{u_n\}\subset \Gamma(a)$ be such that $u_n\rightharpoonup u$ in $C$, as $n\to \infty$. Then, for each $n\in\mathbb N$, by Lemma~\ref{Minty}, we get
$u_n\in K(u_n)$ and
\begin{equation*}
\langle T(a,v),v-u_n\rangle+\varphi(v)-\varphi(u_n)\ge \langle m,v-v_n\rangle
\ \ \textrm{for all}\ \ v\in K(u_n).
\end{equation*}
It follows from hypothesis $(\mathcal{H}_{K})$(ii) and convergence $u_n\rightharpoonup u$ with $u_n\in K(u_n)$ that $u\in K(u)$.
Moreover, for any $z\in K(u)$, using
$(\mathcal{H}_{K})$(i),
there exists a sequence $\{v_n\}\subset C$ such that
$v_n\in K(u_n)$ and $v_n\to z,$ as $n\to\infty.$ Therefore, we have
\begin{equation*}
\langle T(a,v_n),v_n-u_n\rangle+\varphi(v_n)-\varphi(u_n)
\ge \langle m,v_n-v_n\rangle.
\end{equation*}
By the continuity of $\varphi$,
see $(\mathcal{H}_{\varphi})$, and passing to the upper limit, as $n\to \infty$ in the above inequality, we obtain $u\in \Gamma(a)$ proving that $\Gamma(a)$ is weakly closed for each
$a\in A$. The proof of the theorem is complete.
\end{proof}

As a consequence of Theorem~\ref{maintheorem1}, we deduce
the following corollary which is a recent result of~\cite[Theorem 2.2]{khanmotreanu}.
\begin{corollary}\label{corollary}
Let $V$ be a Hilbert space, $m \in V^*$ and
$(\mathcal{H}_K)$ hold.
Assume that there is a bounded subset $C_0$ of $V$ with $K(u)\cap C_0\neq \emptyset$ for all
$u\in C$, and $T\colon B\times V\times V\to \mathbb{R}$ is a trilinear form satisfying the following continuity and coercivity conditions
\begin{align}
|T(a,u,v)|&\leq \beta\|a\|_B\|\|u\|_V\|v\|_V
\ \ \textrm{\rm for all}\ \ (a,u,v)\in B\times V\times V
\ \ \textrm{\rm with}\ \ \beta>0,\\[1mm]
T(a,v,v)&\ge \alpha\|v\|_V^2
\ \ \mbox{\rm for all}\ \ (a,v)\in A\times V
\ \ \textrm{\rm with}\ \ \alpha>0.
\end{align}
Then, for each $a\in A$, the set of solutions, $u\in K(u)$ such that
\begin{equation*}
T(a,u,v-u)\ge \langle m,v-u\rangle
\ \ \mbox{\rm for all}\ \ v\in K(u),
\end{equation*}
is nonempty, bounded, and weakly closed.
\end{corollary}

\begin{remark}
\emph{We note that the proofs of Theorem~\ref{maintheorem1}
and~\cite[Theorem 2.2]{khanmotreanu} are essentially based on the same
fixed point  principle, Theorem~\ref{fpt}. However, in this paper, we deal with inequality (MQVI) in which the mapping $T$ is not necessarily strongly monotone and linear, and also a convex and lower semicontinuous
function appears in the inequality.
This results in the additional difficulty that the variational selection is not a single-valued map, and some important properties obtained in~\cite{khanmotreanu} are not available.}
\end{remark}

\subsection{An Optimization Framework for the Inverse Problem. An Existence Result}
The goal of this subsection is to investigate the inverse problem of identifying a parameter in a nonlinear mixed quasi-variational inequality.

We now recast the inverse problem of parameter identification as the following regularized optimization problem: find $a\in A$ such that
\begin{equation}\label{inverseproblem}
a\in \arg\min_{b\in A}J_\kappa(b),
\end{equation}
where, for
the regularization parameter $\kappa>0$,
the cost functional $J_\kappa\colon B\to \mathbb{R}$ is defined by
\begin{equation*}
J_\kappa(a):=\min_{u\in \Gamma(a)}\frac{1}{2}\|u-z\|_Z^2+\kappa R(a).
\end{equation*}
Here, for $a\in A$, $\Gamma(a)$ represents a set of solutions to \eqref{mqvi}, $Z$ is the data space which we assume to be a real Hilbert space such that $V$ is continuously embedded in $Z$, $z\in Z$ is a given data, and $R$ is the regularization operator.

We introduce the following assumptions.
\begin{enumerate}
\item[($\mathcal{H}_1$).]
Let $L$ and $E$ be two Banach spaces. Assume that the Banach space $B$ is continuously em\-bed\-ded into $L$ and the embedding from $E$ to $L$ is compact. The set $A$, consisting of real-valued functions, is a subset of $B\cap E$, closed and bounded in $B$, and closed in $L$.
\item[($\mathcal{H}_2$).]
Let $\{a_n\}\subset A$, $a\in A$, and
$\{u_n\}$, $\{v_n\}\subset V$. If $\{a_n\}$ is bounded in $B$ and converges strongly to $a$ in $L$, $\{v_n\}$ converges strongly to $v$ in $V$, and $\{u_n\}$ is bounded in $V$, then
\begin{equation*}
\limsup_{n\to\infty}
\langle T(a_n-a,v_n),u_n\rangle \le 0.
\end{equation*}
\item[($\mathcal{H}_3$).]
$R\colon E\to \mathbb{R} $ is convex, lower-semicontinuous with respect to
$\|\cdot\|_L$ and
\begin{eqnarray*}
R(a)\ge \tau_1\|a\|_E-\tau_2
\ \ \mbox{\rm for all $a\in A$ and for some $\tau_1>0$, $\tau_2>0$}.
\end{eqnarray*}
\end{enumerate}

\begin{remark}
\emph{Hypotheses ($\mathcal{H}_1$) and ($\mathcal{H}_3$) have been used in~\cite{khanmotreanu}.
However, assumption ($\mathcal{H}_2$)
is weaker than the following one required in~\cite{khanmotreanu}:
for any sequence
$\{b_k\}\subset B$ with $b_k\to 0$ in $L$, any bounded sequence $\{u_k\}\subset V$,
and fixed $v\in V$, we have
\begin{equation*}
\langle T(b_k,u_k),v\rangle\to 0
\ \ \mbox{as}\ \ k\to\infty.
\end{equation*}
%In fact, the above inequality can imply %ours, ($\mathcal{H}_2$).
Hence, our results are also available for identification of parameters with the regularized output-least-squares problem considered in~\cite{khanmotreanu}.}
\end{remark}

We have the following existence result for the regularized optimization problem~(\ref{inverseproblem}):
\begin{theorem}\label{maintheorem2}
Assume that $(\mathcal{H}_T)$, $(\mathcal{H}_K)$, $(\mathcal{H}_{\varphi})$,  $(\mathcal{H}_0)$--$(\mathcal{H}_3)$ hold,
and $m \in V^*$.
Then, for each $\kappa>0$,
the regularized optimization problem $(\ref{inverseproblem})$ admits a solution.
\end{theorem}

\begin{proof}
It is based on the Weierstrass type theorem
and uses the compactness and lower semicontinuity arguments.
First, we show that the function $J_\kappa\colon B\to \mathbb{R}$ is well-defined. For this, we only need to show that $\min_{u\in\Gamma(a)}g(u):=\min_{u\in\Gamma(a)}\frac{1}{2}\|u-z\|_Z^2$ is well-defined. Since the function $g$ is bounded from below, there exists a minimizing sequence $\{u_n\} \subset \Gamma(a)$ such that
\begin{equation*}
\lim_{n\to\infty}g(u_n)=\inf_{u\in \Gamma(a)}g(u).
\end{equation*}
It follows from Theorem~\ref{maintheorem1} that $\Gamma(a)$ is nonempty, bounded and weakly closed.  The reflexivity of $V$ implies that $\Gamma(a)$ is a weakly compact subset of $V$. Without any loss of generality, we may assume that $u_n\rightharpoonup u^*$, as $n \to \infty$ with $u^*\in \Gamma(a)$. This convergence combined
with the weak lower semicontinuity of $g$
yields
\begin{equation*}
\inf_{u\in \Gamma(a)}g(u)\le g(u^*)\le \liminf_{n\to\infty}g(u_n)=  \lim_{n\to\infty}g(u_n)=\inf_{u\in \Gamma(a)}g(u),
\end{equation*}
which ensures that $J_\kappa$ is well-defined.

Next, by virtue of definition of $J_\kappa$ and hypothesis $(\mathcal{H}_3)$,
we have
\begin{equation*}
J_\kappa(a)=\inf_{u\in \Gamma(a)}g(u)+\kappa R(a)\ge \kappa R(a)\ge\kappa(\tau_1\|a\|_E-\tau_2)\ge -\kappa\tau_2,
\end{equation*}
which implies that $J_\kappa$ is bounded from below. Consequently, there exists a minimizing sequence $\{a_n\}\subset A$ such that
\begin{equation}\label{mina}
\lim_{n\to\infty}J_\kappa(a_n)=\inf_{b\in A}J_\kappa(b).
\end{equation}
By the following estimate
\begin{equation*}
J_\kappa(a_n)\ge \kappa R(a_n)\ge \kappa(\tau_1\|a_n\|_E-\tau_2),
\end{equation*}
we deduce that the sequence $\{a_n\}$ is bounded in $E$. Moreover, the compactness of the embedding
of $E$ into $L$ entails that the sequence $\{a_n\}$ is relatively compact in $L$. Without any loss of generality, we may suppose,
by passing to a subsequence, if necessary, that $a_n\to a$ in $L$, as $n \to \infty$.
Since $A$ is closed in $L$, see $(\mathcal H_1)$, we have $a\in A$.

Subsequently, let $\{u_n\} \subset V$
be a sequence such that
\begin{eqnarray}\label{minu}
u_n\in \Gamma(a_n)\quad\mbox{ and }\quad g(u_n)=\min_{v\in \Gamma(a_n)}g(v).
\end{eqnarray}
By using the uniform coercivity condition \eqref{coercive}, it follows that $\{u_n\}\subset \Gamma(A)$ is bounded as well.
Passing to a subsequence, we may assume that
$u_n\rightharpoonup u$ in $V$, as $n \to \infty$. Since $u_n\in\Gamma(a_n)$, we obtain from Lemma~\ref{Minty} that $u_n\in K(u_n)$ and
\begin{equation}\label{eqn13}
\langle T(a_n,v),v-u_n\rangle+\varphi(v)-\varphi(u_n)\ge \langle m,v-u_n\rangle
\ \ \textrm{for all}\ \ v\in K(u_n).
\end{equation}
The convergence $u_n\rightharpoonup u$ in $V$ and $(\mathcal{H}_K)$\,(ii) imply that $u\in K(u)$. Moreover, for any $w\in K(u)$, hypothesis $(\mathcal{H}_K)$\,(i) allows to choose a sequence $\{v_n\}\subset V$ with $v_n\in K(u_n)$ and $v_n\to w$, as $n\to\infty$.
Putting $v=v_n$ into (\ref{eqn13}) and using hypothesis $(\mathcal{H}_T)$, we obtain
\begin{align*}
\langle m,v_n-u_n\rangle&\le \langle T(a_n,v_n),v_n-u_n\rangle
+\varphi(v_n)-\varphi(u_n)\\[1mm]
&=\langle T(a_n-a,v_n),v_n-u_n\rangle+\langle T(a,v_n),v_n-u_n\rangle +\varphi(v_n)-\varphi(u_n).
\end{align*}
Since $v_n\to w$, $u_n\rightharpoonup u$ in $V$ and
$a_n\to a$ in $L$ { with $a\in A$ and $\{a_n\}\subset A$}, as $n\to\infty$,
we pass to the upper limit in the above inequality, as $n\to \infty$, and use
the continuity of $\varphi$ and
$(\mathcal{H}_2)$ to get
\begin{align*}
\langle m,w-u\rangle
&=\limsup_{n\to\infty} \langle m,v_n-u_n\rangle\\
&\le \limsup_{n\to\infty}\Big(\langle T(a_n-a,v_n),v_n-u_n\rangle+\langle T(a,v_n),v_n-u_n\rangle +\varphi(v_n)-\varphi(u_n)\Big)\\
&\le \limsup_{n\to\infty}\langle T(a_n-a,v_n),v_n-u_n\rangle+\limsup_{n\to\infty}\langle T(a,v_n),v_n-u_n\rangle\\
&\quad+\limsup_{n\to\infty}\varphi(v_n)-\liminf_{n\to\infty}\varphi(u_n)\\
&\le \langle T(a,w),w-u\rangle+\varphi(w)-\varphi(u)\ \,\mbox{ for all }w\in K(u).
\end{align*}
The latter proves that $u\in \Gamma(a)$.

Finally, since $g$ is convex and lower semicontinuous, so, it is also weakly lower-semicontinuous. Therefore, from condition $(\mathcal{H}_3)$, (\ref{mina})
and (\ref{minu}), we have
\begin{align*}
\inf_{b\in A}J_\kappa(b)&\le J_\kappa(a)=\min_{w\in \Gamma(a)}g(w)+\kappa R(a)\le g(u)+\kappa R(a)\\
&\le \liminf_{n\to \infty}g(u_n)+\liminf_{n\to \infty}\kappa R(a_n)\\
&\le\liminf_{n\to \infty}\Big(g(u_n)+\kappa R(a_n)\Big)\\
&=\liminf_{n\to \infty}\Big(\min_{w\in \Gamma(a_n)}g(w)+\kappa R(a_n)\Big)\\
&=\liminf_{n\to \infty} J_\kappa(a_n)=\inf_{b\in A}J_\kappa(b),
\end{align*}
which proves that $a\in A$ is also a solution of the  regularized optimization problem (\ref{inverseproblem}).
The proof of the theorem is complete.
\end{proof}

\section{An Implicit Obstacle Problem of  $p$-Laplacian type}\label{Section3}

In this section we study a regularized optimization problem governed by a nonlinear implicit obstacle problem involving
an operator of $p$-Laplace type.
Given $z\in L^p(\Omega;\mathbb{R}^N)$,
$1 < p < \infty$ and $\kappa>0$,
we consider the following problem.
\begin{problem}\label{problem3}
\emph{Find $a\in A$ such that
		\begin{equation*}
		a\in \arg\min_{b\in A}J_\kappa
		\ \,\mbox{ with }\ \ J_\kappa(a):=\min_{u\in \Gamma(a)}\|\nabla u-z\|_{L^p(\Omega;\mathbb{R}^N)}+\kappa \, TV(a),
		\end{equation*}
		where $\Gamma(a)$ is the solution set of the following nonlinear mixed quasi-variational inequality: given $a\in A$, find $u\in K(u)$ such that
		\begin{align*}
		&\int_\Omega a(x)\big(|\nabla u(x)|^{p-2}\nabla u(x),\nabla v(x)-\nabla u(x)\big)_{\mathbb{R}^N}\,dx+\int_\Omega\phi(v(x))\,dx\\
		&-\int_\Omega\phi(u(x))\,dx\ge \int_\Omega m(x)\big(v(x)-u(x)\big)\,dx
		\ \ \textrm{for all}\ \ v\in K(u).
		\end{align*}
		Here  $m\in L^p(\Omega)$, $\phi\colon \mathbb{R}\to \overline{\mathbb{R}}$ is a convex and lower semicontinuous functional such that $x\mapsto \phi(v(x))$ belongs to $L^1(\Omega)$ for all $v\in C$, and the set of admissible parameters $A$ is defined as
		\begin{equation*}
			A:=\big\{a\in L^\infty(\Omega)\mid 0<c_1\le a(x)\le c_2\ \,\mbox{ for a.e. }x\in\Omega\ \,\mbox{and }\, TV(a)\le c_3\big\}.
	\end{equation*}}
\end{problem}

The functional setting for the above problem is the following.
Assume that $\Omega$ is an open and bounded domain in $\mathbb{R}^N$, $1\le N<\infty$ with sufficiently smooth boundary $\partial \Omega$ and $1<p<\infty$. We introduce the function spaces $V=W^{1,p}_0(\Omega)$, $Z=L^p(\Omega;\mathbb{R}^N)$, and
$B=L^\infty(\Omega)$. Given a positive constant $c_0$ and a Lipschitz continuous function $c\colon \mathbb{R}\to \mathbb{R}$ with $0< c(s)\le c_0$ for all $s\in\mathbb{R}$, we consider a closed convex subset $C$ of $V$ and a set-valued mapping $K\colon C\to 2^C$ defined by
\begin{align*}
&C=\big\{u\in V\mid |\nabla u(x)|\le c_0\ \,\mbox{ a.e. }x\in\Omega\big\},\\[1mm]
&K(u)=\big\{v\in V\mid |\nabla v(x)|\le c(u(x))\ \,\mbox{ a.e. }x\in\Omega\big\},
\end{align*}
where the symbol $|\cdot|$ stands for the Euclidean norm in $\mathbb{R}^N$.
Recall, that for $f\in L^1(\Omega)$,
the total variation of $f$ is defined by
\begin{equation*}
TV(f):=\sup\bigg\{\int_\Omega f(x)\, \textrm{div} g(x)\,dx\mid g\in C^1(\Omega;\mathbb{R}^N)
\mbox{ with }|g(x)|\le 1\mbox{ for all }x\in \Omega\bigg\}.
\end{equation*}
Evidently, if $f\in W^{1,1}(\Omega)$,
then
$$
TV(f)=\|\nabla f\|_{L^1(\Omega;\mathbb{R}^N)}:=\int_\Omega|\nabla f(x)|\,dx.
$$
As usual, for $f\in L^1(\Omega)$, we say that $f$ has bounded variation, if $TV(f)<\infty$.
Also, we recall that the Banach space
\begin{equation*}
BV(\Omega):=\big\{f\in L^1(\Omega) \mid TV(f)<\infty\big\},
\end{equation*}
which is endowed with the norm
\begin{equation*}
\|f\|_{BV(\Omega)}:=\|f\|_{L^1(\Omega)}+TV(f) \ \,\mbox{ for all }\ f\in BV(\Omega).
\end{equation*}

The main result for Problem~\ref{problem3} reads as follows.
\begin{theorem}
For any $\kappa>0$, Problem~\ref{problem3} admits a solution.
\end{theorem}
\begin{proof} 	
We shall use Theorem~\ref{maintheorem2} to prove the solvability of Problem~\ref{problem3}.
We will verify all hypotheses of Theorem~\ref{maintheorem2}.

Denote $E=BV(\Omega)$, $R(\cdot)=TV(\cdot)$, $\varphi(u)=\int_\Omega\phi(u(x))\,dx$ for all $u\in C$, and $L=L^1(\Omega)$.
Let the operator $T\colon B\times V\to V^*$
be defined by
\begin{equation*}
\langle T(a,u),v\rangle=\int_\Omega a(x)\big(|\nabla u(x)|^{p-2}\nabla u(x),\nabla v(x)\big)_{\mathbb{R}^N}\,dx
\ \ \mbox{\rm for all} \ \ u, v\in V.
\end{equation*}
From the above definition, we can readily see that $T$ enjoys hypothesis $(\mathcal{H}_T)$.
On the other hand, it follows
from~\cite[lemma 4.1]{Kano} that the set-valued mapping $K$ fulfills conditions $(\mathcal{H}_K)$. Moreover, by a direct argument, the convexity and lower semicontinuity of $\phi$ implies that functional $\varphi$ is convex and lower semicontinuous, see,
e.g.~\cite[p. 854]{bartosz}.
This combined with assumption that the map
$x\mapsto \phi(v(x))$ is in $L^1(\Omega)$
for all $v\in C$ yields condition  $(\mathcal{H}_{\varphi})$.

Next, we will verify that hypotheses
$(\mathcal{H}_0)$--$(\mathcal{H}_3)$ hold.
Since $C$ is bounded, so, in this case,
we can take $C_0=C$ so that $(\mathcal{H}_{0})$
is satisfied.

Moreover,
$B=L^\infty(\Omega)$ is continuously embedded
into $L= L^1(\Omega)$.
By virtue of definition of $A$,
it is closed in $L=L^1(\Omega)$ and
$A\subset B\cap E$ is bounded and closed in $B$. Furthermore, from the results in~\cite{acar,nashed}, we infer
that the embedding from $E=BV(\Omega)$ to $L^1(\Omega)$ is compact. Therefore, $(\mathcal{H}_1)$ holds.

We will verify condition $(\mathcal{H}_2)$.
Let $\{a_n\}\subset A$ and $a\in A$ be such that $\{a_n\}$ is bounded in $B$ and $a_n\to a$ strongly in $L$, as $n \to \infty$.
Also, let $\{u_n\}$, $\{v_n\}\subset V$ be such that $u_n\to u$ in $V$ and $\{v_n\}$ is bounded in $V$. We have
\begin{align*}
\langle T(a_n-a,u_n),v_n\rangle&=\int_\Omega \big(a_n(x)-a(x)\big)\big(|\nabla u_n(x)|^{p-2}\nabla u_n(x),\nabla v_n(x)\big)_{\mathbb{R}^N}\,dx\\
&\le\int_\Omega |a_n(x)-a(x)||\nabla u_n(x)|^{p-1}|\nabla v_n(x)|\,dx\\
&=\int_\Omega |a_n(x)-a(x)|^{\frac{p-1}{p}}|\nabla u_n(x)|^{p-1}|a_n(x)-a(x)|^{\frac{1}{p}}|\nabla v_n(x)|\,dx.
\end{align*}
Thus, by invoking the H\"older inequality, we obtain
\begin{align*}
&\langle T(a_n-a,u_n),v_n\rangle\le\bigg(\int_\Omega |a_n(x)-a(x)||\nabla u_n(x)|^p\,dx\bigg)^{\frac{p-1}{p}}\bigg(\int_\Omega |a_n(x)-a(x)||\nabla v_n(x)|^p\,dx\bigg)^{\frac{1}{p}}\\
&=\bigg(\int_\Omega |a_n(x)-a(x)||\nabla u_n(x)-\nabla u(x)+\nabla u(x)|^p\,dx\bigg)^{\frac{p-1}{p}}\bigg(\int_\Omega |a_n(x)-a(x)||\nabla v_n(x)|^p\,dx\bigg)^{\frac{1}{p}}.
\end{align*}
We use the elementary inequality  $\big(|\alpha|+|\beta|\big)^p\le 2^{p-1}(|\alpha|^p+|\beta|^p)$
which holds for $\alpha$, $\beta\in\mathbb{R}$
and $p>1$. Then,
\begin{equation}\label{eq2}
\langle T(a_n-a,u_n),v_n\rangle\le M_n\bigg(\int_\Omega 2^{p-1}|a_n(x)-a(x)|\big(|\nabla u_n(x)-\nabla u(x)|^p+|\nabla u(x)|^p\big)\,dx\bigg)^{\frac{p-1}{p}},
\end{equation}
where $M_n$ is defined by $\displaystyle M_n:=\bigg(\int_\Omega |a_n(x)-a(x)||\nabla v_n(x)|^p\,dx\bigg)^{\frac{1}{p}}.$

Since $\{a_n\}$ and $a$ are bounded in $L^\infty(\Omega)$,
$a_n-a\to 0$ in $L=L^1(\Omega)$, $u_n\to u$ in $V$, and $\{v_n\}$ is bounded in $V$,
we deduce that there is a constant $c_4>0$
such that $M_n\le c_4$ for all $n\in\mathbb N$,
and
\begin{equation*}
\int_\Omega 2^{p-1}|a_n(x)-a(x)||\nabla u_n(x)-\nabla u(x)|^p\,dx\to 0\ \,\mbox{ and }\
\int_\Omega 2^{p-1}|a_n(x)-a(x)||\nabla u(x)|^p\,dx
\to 0,
\end{equation*}
as $n\to\infty$. Hence, we obtain $\limsup_{n\to\infty}\langle T(a_n-a,u_n),v_n\rangle\le0.$ Therefore, $(\mathcal{H}_2)$ is satisfied.

Finally, from~\cite[Theorems~2.3 and 2.4]{acar}, \cite{Attouch} and~\cite{nashed},
we know that the functional $f\mapsto TV(f)$ is convex and lower semicontinuous in $L^1(\Omega)$-norm. Also, we have
\begin{equation*}
R(a)=TV(a)=\|a\|_{BV(\Omega)}-\|a\|_{L^1(\Omega)}
\ge \|a\|_{BV(\Omega)}-c_2|\Omega|
\ \ \mbox{\rm for all}\ \ a \in A.
\end{equation*}
This shows that condition $(\mathcal{H}_3)$ is fulfilled with $\tau_1=1$ and $\tau_2=c_2|\Omega|$.

Having verified all the hypotheses,
we are now in a position to apply, Theorem~\ref{maintheorem2} to conclude that Problem~\ref{problem3} has at least one solution $a\in A$. This completes the proof.
\end{proof}

\section{Concluding Remarks}
We have investigated the inverse problem of parameter identification in a nonlinear mixed quasi-variational inequality and applied our results to an implicit obstacle problem of $p$-Laplacian-type.
It would be of natural interest to develop numerical techniques for the inverse problem. For this, we will need to derive optimality conditions for the output-least-squares functional. This is a challenging task since the mapping $T$ is nonlinear and the solution to the inequality is not unique.
Furthermore, it is a nontrivial interesting open question to extend the results of this paper to quasi-hemivariational inequalities. This extension is important in many applications, see~\cite{MOS} and the references therein, where nonconvex potentials are used to model the physical phenomena, and the variational inequality approach is not possible. We intend to carry out the research in this direction in our future work.

\bigskip

\section*{Appendix} Let $X$ be a Banach space with its dual $X^*$,
and $\langle \cdot, \cdot \rangle_{X^*\times X}$
denote the duality pairing between $X$ and $X^*$.
Let $C$ be a nonempty subset of $X$.
We denote by $2^C$ all subsets of the set $C$.

\begin{definition}\label{defap}
\rm
An operator $A\colon X\to X^*$ is called {\it monotone}, if
\begin{eqnarray*}
\langle A(u)-A(v), u-v\rangle_{X^*\times X} \ge 0
\ \ \mbox{\rm for all} \ \ u, v \in X.
\end{eqnarray*}
It is called {\it hemicontinuous}, if for all $u$, $v$, $w\in X$ the functional
\begin{eqnarray*}
t\mapsto\langle A(u+tv),w\rangle_{X^*\times X}
\end{eqnarray*}
is continuous on $[0,1]$.
A set-valued mapping $K\colon C\to 2^C$
is called
$M$-{\it continuous (Mosco continuous)}, if it satisfies the following conditions
\begin{itemize}
\item[$(M_1)$] For any sequence $\{x_n\}\subset C$ with $x_n\rightharpoonup x$, and for each $y\in K(x)$, there exists a sequence $\{y_n\}$ such that $y_n\in K(x_n)$ and $y_n\to y$.
\item[$(M_2)$] For $y_n\in K(x_n)$ with $x_n \rightharpoonup x$ and $y_n\rightharpoonup y$, we have $y\in K(x)$, i.e., the graph of $K$ is sequentially weakly closed.
\end{itemize}
\end{definition}

We recall that a function
$\varphi\colon X\to  \overline{\mathbb{R}}:=\mathbb{R}\cup\{+\infty\}$ is called to be proper, convex and lower semicontinuous, if it fulfills, respectively,
the following conditions
\begin{eqnarray*}
&&
\mbox{dom}(\varphi):=\{u\in X\mid \varphi(u)<+\infty\} \neq\emptyset,
\\[2mm]
&&
\varphi(\lambda u+(1-\lambda)v)\le \lambda \varphi(u) +(1-\lambda)\varphi(v)\ \,\mbox{ for all }\lambda \in[0,1]\ \,\mbox{and }u, v\in X,
\\[2mm]
&&
\varphi(u)\le \liminf_{n\to \infty}\varphi(u_n)\ \,\mbox{ for all sequences }\{u_n\} \subset X \,\mbox{ with }u_n\to u\ {\rm in}\ X.
\end{eqnarray*}

\end{document}